\documentclass[letter,reqno,11pt]{amsart}
\usepackage[margin=1in]{geometry}
\usepackage[utf8]{inputenc}
\usepackage{graphicx, amsfonts, amssymb, amsthm, mathtools,multirow,array}
\usepackage{soul}
\usepackage{hyperref}
\usepackage{enumerate}
\usepackage{verbatim}
\usepackage{tabularx}
\usepackage{colortbl}
\usepackage[nameinlink]{cleveref}
\hypersetup{	colorlinks,	linkcolor={blue},	citecolor={blue},urlcolor={blue}}
\definecolor{Gray}{gray}{0.9}
\newlength\tindent
\newtheorem{theorem}{Theorem}[section]
\newtheorem{prop}[theorem]{Proposition}

\newtheorem{lemma}[theorem]{Lemma}
\newtheorem{conjecture}[theorem]{Conjecture}
\newtheorem{corollary}[theorem]{Corollary}

\numberwithin{equation}{section}

\newcommand{\f}{\mathbb{F}}
\newcommand{\fp}{{\mathbb{F}_p}}
\newcommand{\fq}{{\mathbb{F}_q}}

\newcommand{\fqe}{\mathbb{F}_{q^e}}

\definecolor{purple}{rgb}{0.54, 0.17, 0.89}
\definecolor{dgreen}{rgb}{0.0, 0.5, 0.0}

\title{Permutation binomials of the form  $x^r(x^{q-1}+a)$ over $\mathbb F_{q^e}$}
\author{Ariane M. Masuda, Ivelisse  Rubio, and Javier Santiago}

\address{Department of Mathematics, New York City College of Technology, The City University of New York (CUNY), 300 Jay Street, Brooklyn, NY 11201 USA}

\email{amasuda@citytech.cuny.edu}

\address{Department of Computer Science,
University of Puerto Rico, R\'{i}o Piedras,
17 Ave Universidad, Ste 1701, 
San Juan, Puerto Rico 00925-2537}

\email{ivelisse.rubio@upr.edu, javier.santiago16@upr.edu}

\keywords{Permutation binomial, permutation polynomial, finite field, index}

\begin{document}

\maketitle

\begin{abstract}
       We present several existence and nonexistence results for permutation binomials of the form  $x^r(x^{q-1}+a)$,  where $e\geq 2$ and $a\in \mathbb{F}_{q^e}^*$.  As a consequence, we obtain a complete characterization of such  permutation binomials over $\mathbb{F}_{q^2}$, $\mathbb{F}_{q^3}$, $\mathbb{F}_{q^4}$,  $\mathbb{F}_{p^5}$, and $\mathbb{F}_{p^6}$, where $p$ is an odd prime. 
\end{abstract}

\section{Introduction}\label{intro}

Let $q$ be a prime power and $\mathbb F_q$ denote the finite field of order $q$. Any function from $\mathbb F_q$  to $\mathbb F_q$  is uniquely  represented by a polynomial of degree less than $q$ via the modular reduction $x^q-x$.  A polynomial in $\mathbb F_{q}[x]$ is a {\it{permutation polynomial}} over $\mathbb{F}_q$  if it induces a bijection in  $\mathbb{F}_q$. 

The study of permutation polynomials over finite fields dates back to the 1800s with the work of Hermite~\cite{hermite}. Since then, many researchers have investigated them, not only because they are interesting on their own, but also due to their applications in areas such as coding theory, cryptography, and finite geometry.  Permutation monomials have been completely characterized as $x^r$ permutes $\mathbb{F}_q$ if and only if  $\gcd(r, q-1)=1$. In the absence of a general and simple permutation criterion for binomials, permutation binomials have been the focus of many studies. As a matter of fact, the majority of the results on permutation polynomials concern permutation binomials as they already impose many challenges, despite their simple form.   Given  $x^m+ax^n$ in $\fq[x]$, it is desirable to find simple relationships between $m$, $n$, $a$, and $q$ that make  the binomial  permute $\fq$. This is usually  achieved by   specializing these parameters to particular forms. We refer the interested readers to the surveys~\cite{MR3293406,MR3329981,MR3087321} for an excellent overview of the area. 

Besides the number of terms of the polynomial, there are other parameters  such as the degree, the Carlitz rank, and the index that are relevant in many problems involving polynomials over finite fields. The concept of {\it{index}}  was introduced by Akbary, Ghioca, and Wang~\cite{akbary2009permutation} in 2009; we postpone the precise definition to~\Cref{premi}. It turns out that
any nonconstant polynomial  of index $\ell$ and degree at most $q-1$ in 
$\mathbb{F}_q[x]$ can be uniquely expressed as $\alpha x^rh\left(x^{(q-1)/\ell}\right)+ \beta$.  Wang~\cite{Wa2019} wrote a survey on permutation polynomials classified by their indices, indicating that this might be a promising way of organizing  the many existing results. For instance,  a result by Masuda and Zieve in~\cite{MaZiTransAMS2009} can be stated as: the index of any permutation binomial over the prime field $\fp$ is always less than $\sqrt p -1$.

Wang also suggested a list of open problems in his survey. Problem 7 is ``{\it{Construct and classify permutation polynomials of $\mathbb F_{q^e}$ with intermediate indices such as $q^{e-1}+\cdots+q+1$, $c( q^{e-1}+\cdots+q+1)$, or $(q^{e-1}+\cdots+q+1)/d$, where $c$ is a positive factor of $q-1$, and $d$ is a positive factor of $q^{e-1}+\cdots+q+1$.}}"  Not only there are not many results in this direction, the ones listed in~\cite{Wa2019} for $e\geq 3$ involve polynomials with more than two terms.

We now turn our attention to  permutation binomials of a special type. Let $S$ be the collection of all $(q^e,r,t,a)\in\mathbb N^3\times\mathbb F_{q^e}^*$, $e\geq 2$, such that $x^r(x^{t(q-1)}+a)$ permutes $\mathbb F_{q^e}$. Most of the results in the literature deal with $e=2$ for which the index is  $(q+1)/\gcd (q+1,t)$ from Problem 7.  
In this case,  under the assumptions that $\gcd(rp,t(q-1))=1$, where $p=\mathrm{char}\;\mathbb F_q$ and $(-a)^{(q+1)/\gcd(q+1,t)}\neq 1$,  Hou~\cite{houarxiv2016} grouped  several results that appear in~\cite{MR3293406,MR3071828,  MR3587259, MR3276311, MR3333142,  zarxivredei13} into the following four  families:
\begin{enumerate}[\normalfont(i)]
\item $(q^2,r,t,a)$, $a^{q+1}=1$, $\gcd(r-t,q+1)=1$
\item $(q^2,r,1,a)$,  $r \equiv 1 \pmod{q+1}$
\item $(q^2,1,2,a)$, $(-a)^{(q+1)/2}=3$
\item $(q^2,3,2,a)$, $(-a)^{(q+1)/2}=1/3$
\end{enumerate}
and conjectured that these are the only infinite ones in $S$ with $e=2$.

In this paper we   investigate  conditions for $(q^e,r,1,a)$ to belong to $S$. In other words, we consider
 binomials of the form 
$f(x)=x^r\left(x^{q-1}+ a\right)$  over $\mathbb F_{q^e}$ for $e\geq 2$. When $r \leq q^e-q$, the index of $f(x)$ is $\ell=q^{e-1}+\cdots+q+1$, so our results contribute to the solution of Problem 7.
An immediate necessary condition is that $(-a)^{\ell}\neq 1$, since otherwise $f(x)$  would have more than one root. Hou's second family concerns the case $e=2$, and
Liu studied the case $e=3$. We state their results  as follows.

\begin{prop}
 \label{liliu}
Let $f(x) = x^r(x^{q-1} + a) \in \mathbb{F}_{q^e}[x]$ with $e\in\{2,3\}$ and $a\neq 0$, and let
$\ell=q^{e-1}+\cdots+q+1$. 
\begin{enumerate}[\normalfont(i)]
    \item  \cite[Theorem 4.2]{houarxiv2016}  Suppose $e=2$. Then $f(x)$ permutes $\mathbb F_{q^2}$ if and only if  $a^{\ell}\neq 1$, $\gcd(r,q-1)=1$, and $r \equiv 1 \pmod{\ell}$.

    \item  \cite[Theorem 1]{LiuArX2019}
     Suppose $e=3$, $q$ is odd, and $r\leq \ell$. Then $f(x)$ permutes  $\mathbb{F}_{q^3}$ if and only if $(-a)^{\ell} \ne 1$ and $r = 1$.  
\end{enumerate}
\end{prop}

Li et al.~\cite[Theorem 3.1]{LiQuChFFA2017} showed a weaker version of~\Cref{liliu}(i) by considering $r \leq \ell$, which is the same restriction that appears in (ii). In both cases, the authors restricted $r$ to such values due to the following result. 

\begin{prop}\cite[Lemma 2.2]{LiQuChFFA2017}\label{rootsCor} Let $f(x) = x^{r}(x^{(q^e-1)/\ell}+a)\in\mathbb F_{q^e}[x]$ with   $a\neq 0$ and $\ell\mid q^e-1$. If $f(x)$ permutes $\mathbb F_{q^e}$ and $\gcd(r+\ell,(q^e-1)/\ell)=1$, then $x^{r+\ell}(x^{(q^e-1)/\ell}+a)$ also permutes $\mathbb F_{q^e}$.
\end{prop}

We note that~\Cref{liliu}(ii) together with~\Cref{rootsCor} may not produce all permutation binomials $x^r(x^{q-1}+a)$ over $\mathbb{F}_{q^3}$. For instance, the  permutation binomials of  this form with degree less than $26$ over $\mathbb{F}_{27}$ are $x(x^2+a)$ and $x^{23}(x^{2}+a)$, with $(-a)^{13}\neq 1$.  While the former one is the only permutation binomial given by~\Cref{liliu}(ii), the latter one is not given by~\Cref{rootsCor}. In fact, ~\Cref{rootsCor}  cannot be applied to $x(x^2+a)$ since $\gcd(r+\ell, (q^e-1)/\ell) = \gcd(14,2)=2$. Moreover, we observe that, in order to obtain $x^{23}(x^{2}+a)$ from~\Cref{rootsCor}, we would need $x^{10}(x^2+a)$ to permute $\mathbb{F}_{27}$, which is clearly not true. In particular, this shows that the converse of~\Cref{rootsCor} does not hold. One of our goals in this paper is to extend~\Cref{liliu}(ii) to include the values of $r$ such that $r > \ell$.

We show several existence and nonexistence results of permutation binomials of the form $x^r(x^{q-1} + a)$ over $\mathbb F_{q^e}$ for  $e\geq 2$. As a consequence, we completely characterize such permutation binomials over $\mathbb F_{q^2}, \mathbb F_{q^3}, \mathbb F_{q^4}$, $\mathbb F_{p^5}$, and $\mathbb F_{p^6}$, where $p$ is an odd prime. In particular, we extend~\Cref{liliu}. Even though the case  $\mathbb F_{q^2}$  has been resolved by Hou, we include it in our characterization, since we state our result in terms of $\ell$, 
which allows us to have a statement similar to~\Cref{liliu}(i) that holds for other values of $e$. Below we summarize the  results. We denote the (nonnegative) reduction of $r$ modulo $\ell$ by $r \pmod{\ell}$.

\begin{theorem}\label{e2345}
Let $f(x) = x^r(x^{q-1} + a) \in \mathbb{F}_{q^e}[x]$ with $2\leq e \leq 6$ and  $a\neq 0$,  and  let
$\ell=q^{e-1}+\cdots+q+1$.
\begin{enumerate}[\normalfont(i)]
    \item When $e=2,3,4$, $f(x)$ permutes $\mathbb{F}_{q^e}$ if and only if $(-a)^{\ell} \ne 1$,  $\gcd(r, q-1) = 1$, and $r \pmod{\ell}\in\{1, \ell-q\}$.
    \item When $e=5$ and $q$ is an odd prime, $f(x)$ permutes $\mathbb{F}_{q^5}$ if and only if $(-a)^{\ell} \ne 1$,  $\gcd(r, q-1) = 1$, and $r\pmod{\ell}\in\{1, \ell-q, q^3+1,q^4+q^2+1\}$.
    \item When $e=6$ and $q$ is an odd prime, $f(x)$ permutes $\mathbb{F}_{q^6}$ if and only if $(-a)^{\ell} \ne 1$,  $\gcd(r, q-1) = 1$, and $r\pmod{\ell}\in\{1, \ell-q\}$.
\end{enumerate}
\end{theorem}

In the next result we construct  permutation binomials $f(x)$ for arbitrary $e$.

\begin{theorem}\label{geral}
Let $f(x) = x^r(x^{q-1} + a)\in\mathbb{F}_{q^e}[x]$ with $e\geq 2$ and $a\neq 0$, and let $\ell=q^{e-1}+\cdots +q+1$. Then $f(x)$ permutes $\fqe$ and is the composition of a linearized binomial and a monomial if and only if $(-a)^\ell \neq 1$ and $r = s\ell + \sum_{i=0}^{k-1}q^{hi}  \pmod{q^e-1}$, where $\gcd(h, e) = 1$, $k \pmod{e} = h^{-1}$, $s$ is a positive integer, and  $\gcd(r, q-1) = 1$. 
\end{theorem}

\Cref{geral} shows that, for any $q^e$ with $e\geq 2$, there is always a $4$-tuple $(q^e,r,1,a)$ in $S$. So this yields an infinite family in $S$ that extends Hou's second family described in~\Cref{liliu}(i).  Moreover, we have run a computer search for $q^e < 10^8$ that leads us to conjecture that our family is the only one  in $S$ for which $t=1$.

\begin{conjecture}\label{conj}
Let $f(x) = x^r(x^{q-1} + a)\in\mathbb{F}_{q^e}[x]$ with $e\geq 2$ and $a\neq 0$,  and let
$\ell=q^{e-1}+\cdots +q+1$.
Then $f(x)$ permutes $\fqe$ if and only if $f(x)$ is congruent to the composition of a linearized binomial $L(x)=x^{q^h}+ax$ and the monomial $x^r$ modulo $x^{q^e}-x$, where  $\left(-a\right)^{\ell}\neq 1$   and $\gcd(r,q-1)=1$. \end{conjecture}

\Cref{e2345,geral} 
show that this conjecture is true when $e\in\{2,3,4\}$ for any $q$, and when $e\in\{5,6\}$ for any odd prime $q$. One direction of~\Cref{e2345} is obtained through a series of nonexistence results. In addition, these results collectively restrict the possible values of $r$  for the remaining values of $q^e$.

We organize this paper as follows.  In~\Cref{premi} we present some tools that we use in~\Cref{nece}, including a permutation criterion and some technical results involving binomial coefficients. The goal of~\Cref{suff} is to show~\Cref{geral}. In~\Cref{nece} we present several nonexistence results.
Finally, in~\Cref{special} we apply some results from ~\Cref{suff,nece} to show ~\Cref{e2345} as a combination of Propositions ~\ref{q234final} to~\ref{p6}. We end the paper with a discussion on~\Cref{conj}.

\section{Preliminaries}\label{premi}

We start this section by defining the index of a polynomial $f(x)\in\mathbb F_{q^e}[x]$ of degree at most $q^e-1$.  Write 
$$f(x) = \alpha(x^d+a_{d-i_1}x^{d-i_1}+\cdots + a_{d-i_k}x^{d-i_k})+\beta,$$
where $\alpha,a_{d-i_j}\neq 0$ for $j=1,\ldots,k$. We suppose that $k\geq 1$ and $d-i_k=r$.  The integer 
$$\ell = \dfrac{q^e-1}{\gcd(d-r,d-r-i_1,\dots,d-r-i_{k-1},q^e-1)}$$
 is called the {\emph{index}} of $f(x)$. We note that when $r \leq q^e-q$ the index of $f(x)=x^r\left(x^{q-1}+a\right)\in\mathbb F_{q^e}[x]$ is $\ell=q^{e-1}+\cdots+q+1$.

Our results rely on  a variation of Hermite's criterion obtained by Masuda, Panario, and Wang, namely Theorem 1 in~\cite{MPW}. The authors specialized their criterion for polynomials to binomials. The result is immediate, and it appears as Corollary 2 without a proof. In both cases, they require the degree of the polynomial in $\mathbb{F}_{q^e}[x]$ to be less than $q^e-1$.  
We examined the proof of Theorem 1, and checked that this condition is not needed. Below we state their criterion for permutation binomials without the degree requirement.  

\begin{prop}\cite[Corollary 2]{MPW}\label{MPWbin}
Let $f(x) = x^r(x^s + a) \in \mathbb{F}_{q^e}[x]$ with   $q^e \geq 3$, $a \ne 0$, and $r,s\geq 1$. Then $f(x)$ permutes $\mathbb{F}_{q^e}$ if and only if 
\begin{equation}\label{sum}
    \sum_{A \in S_N} \binom{N}{A}a^{N-A} = 
    \begin{cases}
    0  & \text{ if }  N = 1, \dots, q^e-2 \\
    1 & \text{ if }  N = q^e-1,
    \end{cases}
\end{equation}
where
\begin{equation*}
    S_{N} = \bigg\{ A \in \mathbb{Z} \colon A = \frac{j(q^e-1) -rN}{s}, \text{ where } j \in \mathbb{Z}  \text{ and }  0 \leq A \leq N \bigg\}.
\end{equation*}
\end{prop}

For convenience, we may write $S_{N,r}$, $A_j$ or $A_{j,r}$ to emphasize the parameters involved.  For binomials $f(x)=x^r(x^{q-1}+a)$  $\in\mathbb{F}_{q^e}[x]$ and $\ell = q^{e-1}+\cdots+q+1$, we have
\begin{equation}\label{SNr}S_N=\bigg\{ A_j \in \mathbb{Z} \colon A_j = j \ell -  \frac{rN}{q-1},   \text{ where } j \in \mathbb{Z}  \text{ and } 0 \leq A_j \leq N \bigg\}.\end{equation}
The next lemma provides the only two possible values for $|S_N|$  when $rN/(q-1)$ is an integer. The proof gives the exact conditions for each value to occur. Since we do not use these conditions, for simplicity, we opt to state the result without them.

\begin{lemma}\label{count}
Let $f(x) =  x^r(x^{q-1}+ a)\in \mathbb{F}_{q^e}[x]$ with $e\geq 2$ and $a\neq 0$, and let 
$\ell=q^{e-1}+\cdots+q+1$.
If  $q-1\mid rN$ for some $N\in\{1,\dots, q^e-1\}$, then $|S_{N}|=\left\lfloor N/\ell\right\rfloor$ or $\left\lfloor N/\ell\right\rfloor +1$. 
\end{lemma}

\begin{proof}
By~\eqref{SNr}, we have  
\begin{align*}
    |S_{N}| &=  \left|\bigg\{ j \in \mathbb{Z} \colon 0 \leq j\ell-\frac{rN}{q-1}\leq N \bigg\}\right|\\
     &= \left|\bigg\{ j \in \mathbb{Z} \colon 
    \left\lceil\frac{rN}{\ell(q-1)}\right\rceil \leq j\leq  \left\lfloor\frac{rN}{\ell(q-1)} +\frac{N}{\ell}\right\rfloor \bigg\}\right|.
\end{align*}
The result follows from the fact that for $u,v \geq 0$ we have 
$$
\left|  \left[\lceil u\rceil, \lfloor u+v \rfloor\right] \cap \mathbb Z  \right|=
\begin{cases}
\left| [ u,  u + \lfloor v \rfloor] \cap \mathbb Z  \right| = \lfloor v \rfloor +1 &\text{if } u\in\mathbb Z\\
\left| [\lfloor u\rfloor + 1, \lfloor u \rfloor +\lfloor v \rfloor] \cap \mathbb Z  \right| = \lfloor v \rfloor &\text{if } u\not\in\mathbb Z, u- \lfloor u\rfloor + v - \lfloor v \rfloor < 1\\
\left| [\lfloor u\rfloor + 1, \lfloor u \rfloor +\lfloor v \rfloor +1] \cap \mathbb Z  \right| = \lfloor v \rfloor +1 &\text{if } u\not\in\mathbb Z, u- \lfloor u\rfloor + v - \lfloor v \rfloor \geq 1.
\end{cases}
$$
\end{proof}

Following the idea underlying~\Cref{rootsCor} with regard to reducing the possible values of $r$, we show that   it suffices to consider $r\in\{1,\dots,\ell\}$   to show the nonexistence of permutation binomials in a special circumstance. This result will play an important role in our proofs in~\Cref{nece}.

\begin{lemma} \label{rlargerindex}
Let $e\geq 2$, $a\in\fqe^*$, and  $\ell=q^{e-1}+\cdots+q+1$. Let $r$ and $s$ be positive integers such that 
$r\equiv s\pmod{\ell}$. Suppose that $x^s(x^{q-1} +a)$ does not permute $\mathbb{F}_{q^e}$. If $\sum_{A \in S_{N, s}}\binom{N}{A}a^{N-A} \ne 0$ for some  $N\in\{1,\dots,q^e-2\}$ such that $q-1\mid N$, then $x^r(x^{q-1} + a)$ does not permute $\fqe$.
\end{lemma}

\begin{proof}
Let  $A_{j,r}\in S_{N,r}$. Write $r=s+k\ell$,  where $k$ is an integer, and set $j' = j - kN/(q-1)$. Then 
\begin{align*}A_{j,r} &= j\ell - r\frac{N}{q-1} \\  
                & = \left(j' +  k\frac{N}{q-1}  \right)\ell - (s+k\ell)\frac{N}{q-1} \\ 
                & = j' \ell -  s\frac{N}{q-1} \\
                &= A_{j',s}  \in S_{N,s},
\end{align*}
so $S_{N,r}\subseteq S_{N,s}$. This also shows that $S_{N,s}\subseteq S_{N,r}$.
Therefore $S_{N,r}=S_{N,s}$. By~\Cref{MPWbin}, we conclude that $x^r(x^{q-1}+a)$
does not permute $\fqe$.
\end{proof}

The next two results concern identities with binomial coefficients. Our convention is that $\binom{n}{k}=0$
if $n<k$.

\begin{prop}[Lucas' Theorem] \label{Lucas}Let $n\geq k$ be positive integers and $p$ be a prime. If $n=n_mp^m+n_{m-1}p^{m-1}+\cdots+n_0$ and $k=k_mp^m+k_{m-1}p^{m-1}+\cdots+k_0$ are the $p$-adic expansions of $n$ and $k$, then \begin{equation}
\label{lucas}
\binom{n}{k}\equiv \binom{n_m}{k_m}\cdots\binom{n_0}{k_0} \pmod{p}.
\end{equation}
\end{prop}

When $q$ is a power of $p$, the identity~\eqref{lucas} also holds with  the $q$-adic  expansions of $n$ and $k$, namely, $n=n_mq^m+n_{m-1}q^{m-1}+\cdots+n_0$ and $k=k_mq^m+k_{m-1}q^{m-1}+\cdots+k_0$.

\begin{lemma} \label{binom} Let $q=p^m$, where $p$ is prime, and let $k$ be an integer such that $0\leq k \leq q-1$. Then
\begin{enumerate}[\normalfont(i)]
\item $\displaystyle\binom{q-1}{k}\not\equiv 0 \pmod{p}$,
\item $\displaystyle\binom{q-2}{k} \equiv 0 \pmod{p}$ if and only if $k \equiv p-1 \pmod{p}.$
\end{enumerate}
\end{lemma}

\begin{proof}
Let $k=k_{m-1}p^{m-1}+\cdots+k_0$ be the $p$-adic expansion of $k$. 
\begin{enumerate}[\normalfont(i)]
\item  Write $q-1=(p-1)p^{m-1}+\cdots+(p-1)p+(p-1)$. By Lucas' Theorem,  $\binom{q-1}{k} \equiv \binom{p-1}{k_{m-1}}\cdots  \binom{p-1}{k_0} \not\equiv 0 \pmod{p}$, since $0\leq k_i < p$.
\item Write $q-2=(p-1)p^{m-1}+\cdots+(p-1)p+(p-2)$. By Lucas' Theorem,  $\binom{q-2}{k} \equiv \binom{p-1}{k_{m-1}}\cdots \binom{p-1}{k_{1}}  \binom{p-2}{k_0} \pmod{p}$. Since $\binom{p-1}{k_i}\not\equiv 0 \pmod{p}$ for each $i$, we have $\binom{q-2}{k} \equiv 0 \pmod{p}$ if and only if 
$\binom{p-2}{k_0} \equiv 0 \pmod{p}$, which holds if and only if $k_0=p-1$.
\end{enumerate}
\end{proof}

\section{Existence results}\label{suff}

In this section we construct permutation binomials of the form $f(x) = x^r(x^{q-1} + a)$ 
over $\f_{q^e}$.  Two immediate necessary conditions are that $\gcd(r,q-1)=1$ and $(-a)^{\ell}\neq 1$, where $\ell=q^{e-1} + \cdots + q + 1$. The former condition prevents us from writing $f(x)$ as a composition of a binomial and a  permutation monomial $x^n$, where $n=\gcd(r,q-1)$.  Next we consider a more general version of the latter condition, and show that it is equivalent to having a family of linearized binomials that  permute $\fqe$. The result is not new, but we include a proof for completeness.

\begin{lemma}\label{lpb1}
Let $e \geq 2$ and  $a \in \mathbb{F}_{q^e}^*$. Then $L(x) = x^{q^h} + ax$ permutes $\mathbb{F}_{q^e}$ if and only if $(-a)^{(q^e-1)/(q^d-1)} \neq 1$, where $d = \gcd(e, h)$.
\end{lemma}

\begin{proof}
The linearized binomial $L(x)$ permutes $\mathbb{F}_{q^e}$ if and only if zero is its only root  in $\mathbb{F}_{q^e}$. Let $\xi$ be a primitive element of $\mathbb{F}_{q^e}$. Suppose that $L(\xi^k) = 0$ for some integer $k$, so $-a=\xi^{k(q^h-1)}$. Then $(-a)^{(q^e-1)/(q^d-1)} =  \left(\xi^{q^e-1} \right)^{k(q^h-1)/(q^d-1)} = 1$. Conversely, if $(-a)^{(q^e-1)/(q^d-1)} = 1$, then $-a = \xi^{r(q^d-1)}$ for some integer $r$. Since $\gcd(q^e-1, q^h-1) = q^d-1$, we write $q^d-1=b(q^e-1) + c(q^h-1)$ for some integers $b$ and $c$. Then $L(\xi^{rc}) = \xi^{rc}\left(\xi^{rc(q^h-1)} + a\right) = \xi^{rc} \left(\xi^{ r(q^d-1)-rb(q^e-1)} + a \right) = 0$.
\end{proof}

Next we show that any binomial  $f(x)=x^{r}(x^{q-1}+a)\in\fqe[x]$ that is congruent to the composition of a linearized binomial and a monomial  modulo $x^{q^e}-x$ satisfies $f(x)\equiv L(x^r)\pmod{x^{q^e}-x}$.

\begin{lemma}\label{Comp}
Let  $f(x)= x^{r}(x^{q-1}+a)\in\mathbb F_{q^e}[x]$ with $e \geq 2$ and  $a \neq 0$. Then $f(x)\equiv \bar L(x^{s})\pmod{x^{q^e}-x}$, where $\bar L(x)=bx^{q^m}+cx^{q^n}$ for positive integers   $s$, $m$, and $n$  with $n<m<e$ if and only if  $f(x)\equiv L(x^{r})\pmod{x^{q^e}-x}$,  where $L(x)=x^{q^h}+ax$ for some positive integer $h$.
\end{lemma}

\begin{proof}

We have $f(x) \equiv \bar L(x^{s})\pmod{x^{q^e}-x}$, that is,
$x^{r+q-1}+ax^r\equiv bx^{sq^m}+cx^{sq^n}\pmod{x^{q^e}-x}$, which holds in one of the following situations:

\begin{itemize}
\item[(1)] $b=1$, $a=c$, $r+q-1 \equiv sq^m \pmod{q^e-1}$, and $r \equiv sq^n \pmod{q^e-1}$, or
\item[(2)] $c=1$, $a=b$, $r+q-1 \equiv sq^n \pmod{q^e-1}$, and $r \equiv sq^m \pmod{q^e-1}$.
\end{itemize}
 \noindent We want to show that there exists a positive integer $h$ such that $rq^h\equiv r+q-1\pmod {q^e-1}$. It suffices to choose $h=m-n$ in Case (1),  and $h=e+n-m$ in Case (2). This concludes the proof as the converse follows immediately.
\end{proof}

In order to find all binomials $f(x)=x^{r}(x^{q-1}+a)$ that are congruent to compositions of $L(x)=x^{q^h}+ax$ and $x^r$ modulo $x^{q^e}-x$, we need to solve 
$r(q^h-1) -q+1 \equiv 0 \pmod{q^e-1}$ for $h$ and $r$.

\begin{lemma} \label{Congr}
Let $e \geq 2$ and $\ell = q^{e-1}+\cdots +q+1$. Then $r(q^h-1) -q+1 \equiv 0 \pmod{q^e-1}$ if and only if $\gcd( h,e) = 1$ and $r \equiv s\ell + \sum_{i=0}^{k-1}q^{hi}  \pmod{q^e-1}$,
 where $k \pmod{e} = h^{-1}$ and $s$ is any integer.
\end{lemma}

\begin{proof}
We first observe that $r \equiv s\ell + \sum_{i=0}^{k-1}q^{hi}  \pmod{q^e-1}$
satisfies
\begin{align*} \label{congruence}
r(q^{h}-1)-q+1 \equiv & s\ell(q^h-1)+\displaystyle\sum_{i=1}^{k}q^{h i}- \displaystyle\sum_{i=0}^{k -1}q^{h i}-q+1 \\
\equiv & q^{hk}-q \equiv 0 \pmod{q^e-1}.
\end{align*}
Conversely, suppose that $r(q^h-1) -q+1 \equiv 0 \pmod{q^e-1}$. So $r\left(\frac{q^h-1}{q-1}\right) \equiv 1 \pmod{\frac{q^e-1}{q-1}}$ implies that $\gcd\left( \frac{q^h-1}{q-1},\frac{q^e-1}{q-1}\right) = 1$, that is, $\gcd(h,e) = 1$.  Let $\Bar{r} = \sum_{i=0}^{k-1}q^{hi}  \pmod{q^e-1}$. Then, similarly to the calculations above, $\Bar{r}(q^h-1)-q+1 \equiv 0 \pmod{q^e-1}$, so $\Bar{r}$ is also an inverse of $\frac{q^h-1}{q-1}$ modulo $\ell$. Thus $r \equiv \Bar{r} \pmod{\ell}$, i.e., $r \equiv s\ell + \sum_{i=0}^{k-1}q^{hi}  \pmod{q^e-1}$  for any integer $s$.
\end{proof}

\Cref{Compos} is our main tool to construct permutation binomials of the form $x^r(x^{q-1}+a)$ over $\fqe$.  We find all $r$'s that give permutation binomials that are compositions of a linearized binomial and a monomial. For these values of $r$, we require $r$  to be coprime to $q-1$ as we can show that this implies that $r$ is coprime to $q^e-1$.

\begin{theorem}\label{Compos} Let $f(x) = x^r(x^{q-1} + a)\in\mathbb F_{q^e}[x]$ with $e \geq 2$ and $a\neq 0$, and let  $\ell = q^{e-1}+\cdots +q+1$. Then $f(x)$ 
permutes $\fqe$ and is the composition of a linearized binomial and a monomial if and only if $(-a)^\ell \neq 1$ and $r = s\ell + \sum_{i=0}^{k-1}q^{hi}  \pmod{q^e-1}$, where $\gcd(h, e) = 1$, $k \pmod{e} = h^{-1}$, $s$ is a positive integer, and  $\gcd(r, q-1) = 1$.
\end{theorem}

\begin{proof} The combination of the three previous lemmas implies that we only need to show that $\gcd(r, q^e-1)=1$ to complete the `if' direction.
By~\Cref{Congr}, we have $r(q^h-1) -q+1 \equiv 0 \pmod{q^e-1}$. By writing $r(q^{h}-1)-q+1 =(q^{e}-1)t$ for an integer $t$ and dividing both sides by $q-1$, we obtain that $r\sum_{i=0}^{h -1}q^i-t\ell=1$, so $\gcd(r,\ell)=1$, which means that $\gcd(r,q^e-1)=1$. 
\end{proof}

\begin{table}
\centering
\begin{tabular}{|l |l| l| l| l| l|} 
 \hline
 $e$ & $\ell$ & $h$ & $k$ & $r$ & $r\pmod{\ell}$ \\ [0.5ex] 
 \hline\hline
 2 & $q+1$ & 1 & 1& $s\ell+1$ &1\\ 
 \hline
 \multirow{2}{0.5em} {$3$} &  \multirow{2}{4em}{$q^2+q+1$} & 1 & 1& $s\ell+1$ & 1\\ \cline{3-6}
  &  &  2 & 2 & $s\ell+q^2+1$ & $\ell-q$\\ 
 \hline
 \multirow{3}{0.5em} {4} &  \multirow{3}{4em}{$q^3+q^2+q+1$} & 1 & 1 & $s\ell + 1$& $1$\\ \cline{3-6}
  & & 3 & 3 & $s\ell + q^6+q^3+1$& \\
 &  &  &  & $\equiv  s\ell + q^3+q^2+1\pmod{q^4-1}$  & $\ell-q$ \\ [1ex] 
 \hline
\multirow{4}{0.5em} {$5$} &  \multirow{4}{4em}{ $q^4+q^3+q^2+q+1$} & 1&1&$s\ell+1$& $1$\\ \cline{3-6}
 & & 2&3& $s\ell+q^4+q^2+1$& $q^4+q^2+1$\\ \cline{3-6}
 & & 3&2&  $s\ell+q^3+1$ & $q^3+1$\\ \cline{3-6}
 &  & 4&4&  $s\ell+q^{12}+q^8+q^4+1$& \\
 &  & &&  $\equiv s\ell+q^{4}+q^3+q^2+1\pmod{q^5-1}$& $\ell-q$\\
\hline
 \multirow{3}{0.5em}{6} & \multirow{3}{4em}{ $q^5+q^4+q^3+q^2+q+1$} & 1&1&$s\ell+1$& $1$\\ \cline{3-6}
 &  & 5&5&$s\ell+q^{20}+q^{15}+q^{10}+q^5+1$& \\
  &  & &&$\equiv s\ell+q^{5}+q^{4}+q^{3}+q^2+1\pmod{q^6-1}$& $\ell-q$\\
   &&&&&\\
\hline
 \multirow{9}{0.5em}{7} &\multirow{9}{4em} {$q^6+q^5+q^4+q^3+q^2+q+1$} & 1&1&$s\ell+1$& $1$\\
 \cline{3-6}
 & & 2&4&$s\ell+q^6+q^4+q^2+1$& $q^6+q^4+q^2+1$\\
 \cline{3-6}
 & & 3&5&$s\ell+q^{12}+q^9+q^6+q^3+1$& \\ 
 & & &&$\equiv s\ell+q^6+q^5+q^3+q^2+1\pmod{q^7-1}$& $\ell-q^4-q$\\ \cline{3-6}
 & & 4&2&$s\ell+q^4+1$& $q^4+1$\\ \cline{3-6}
 & & 5&3&$s\ell+q^{10}+q^5+1$& \\ 
  & & &&$\equiv s\ell+q^5+q^3+1 \pmod{q^7-1}$& $q^5+q^3+1$\\ \cline{3-6}
 & & 6&6&$s\ell+q^{30}+q^{24}+q^{18}+q^{12}+q^6+1$& \\
  & & &&$\equiv s\ell+q^{6}+q^{5}+q^{4}+q^{3}+q^2+1\pmod{q^7-1}$& $\ell-q$\\
\hline
\multirow{7}{0.5em}{8} &\multirow{7}{4em} {$q^7+q^6+q^5+q^4+q^3+q^2+q+1$} & 1&1&$s\ell+1$& $1$\\ \cline{3-6}
 & & 3&3&$s\ell+q^{6}+q^3+1$& $q^{6}+q^3+1$  \\ \cline{3-6}
 & & 5&5&$s\ell+q^{20}+q^{15}+q^{10}+q^5+1$& \\
 &&&&$\equiv s\ell+q^7+q^5+q^4+q^2+1\pmod{q^8-1}$ & $\ell-q^6-q^3-q$\\ \cline{3-6}
 & & 7&7&$s\ell+q^{42}+q^{35}+q^{28}+q^{21}+ q^{14}+q^7+1$& \\
 &&&&$ \equiv \;s\ell + q^{7}+q^{6}+q^{5}+q^{4}+ q^{3}+q^2+1$& $\ell-q$\\
 &&&&$\pmod{q^8-1}$& \\
  \hline
  \end{tabular}
\caption{Values of $r$ given by~\Cref{Compos} and their reductions modulo $\ell$, for 
$2\leq e\leq 8$.} 
\label{rvalues}
\end{table}

\Cref{rvalues} displays the values of   $r$ given by  \Cref{Compos}, for $2\leq e\leq 8$. We also list their corresponding reductions modulo $\ell$ to illustrate that they are the same values that appear in~\Cref{e2345}, for $2\leq e \leq 6$, showing  one direction of our characterization. 
This direction can be also obtained from the next result that is more general. The other direction is proved in~\Cref{special}.

\begin{corollary}\label{part}
Let $f(x) = x^r(x^{q-1} + a) \in \mathbb{F}_{q^e}[x]$  with $e\geq 2$ and $a\neq 0$, and let $\ell = q^{e-1}+\cdots +q+1$. Suppose that $(-a)^{\ell} \ne 1$ and $\gcd(r, q-1) = 1$.
\begin{enumerate}[\normalfont(i)]
    \item If   $r\pmod{\ell}\in\{1,\ell-q\}$, then $f(x)$ permutes $\mathbb{F}_{q^e}$.
    \item  If $e$ is odd and $r \pmod{\ell}\in\{q^{(e+1)/2}+1,\ell-q^{(e+1)/2}-q\}$,  then $f(x)$ permutes $\mathbb{F}_{q^e}$.
    \end{enumerate}
\end{corollary}

\begin{proof}
\begin{enumerate}[\normalfont(i)]
\item Apply~\Cref{Compos} with $h=1$ and $h=e-1$
\item Apply~\Cref{Compos}  with $h=(e+1)/2$ and $h=(e-1)/2$.
\end{enumerate}
\end{proof}

When $e=3$, the two sets provided by~\Cref{part} are the same, i.e., $\{1,\ell-q\} = \{q^{(e+1)/2}+1,\ell-q^{(e+1)/2}-q\}$.

\section{Nonexistence results} \label{nece}

In this section we use~\Cref{MPWbin} to show the nonexistence of permutation binomials of the form $f(x)=x^r(x^{q-1}+a)$ over $\mathbb{F}_{q^e}$ for $e\geq 2$ and $q\neq 2$. When $q=2$, the binomial takes the form $f(x)=x^r(x+a)$, so it does not permute $\mathbb F_{2^e}$ if $a\neq 0$.  To calculate the sum~\eqref{sum} with $A_j$ running over $S_N$~\eqref{SNr}, we apply Lucas' Theorem using $q$-adic expansions. 
Our first nonexistence result narrows down the possibilities for  $f(x)$ to those with $r \pmod{\ell} = hq+1$ for some integer $h$.

\begin{prop}
\label{rth} 
Let $f(x) = x^r(x^{q-1} + a) \in \mathbb{F}_{q^e}[x]$ with $e\geq 2$ and $a\neq 0$, and let $\ell=q^{e-1}+\cdots+q+1$. If $r \pmod{\ell}\neq hq+1$ for any integer $h$, then $f(x)$ does not permute $\f_{q^e}$.
\end{prop}

\begin{proof}
First, assume that $1 \leq r < \ell$ and $r \ne hq + 1$ for $0 \leq h < \sum_{i=0}^{e-2}q^i$. Then $r \neq  1, q+1, 2q+1, \ldots, \left( \sum_{i=0}^{e-2}q^i \right)q+1$. In other words,  $h_0q+2 \leq r \leq (h_0+1)q$ for some $0 \leq h_0 \leq \sum_{i=1}^{e-2}q^i$. We claim that  $S_N=\{A_{h_0+1}\}$ for $N=q^{e-1}-1$.
By~\Cref{count}, since $N/\ell=(q^{e-1}-1)/(\sum_{i=0}^{e-1}q^i) <1$,   it follows that $|S_N|\leq 1$. We have
$$A_{h_0+1}=(h_0+1)\sum_{i=0}^{e-1}q^i - r \sum_{i=0}^{e-2}q^i.$$ 
\noindent
Since $r \leq (h_0+1)q$,  $$A_{h_0+1} \geq (h_0+1)\sum_{i=0}^{e-1}q^i - (h_0+1) \sum_{i=1}^{e-1}q^i = h_0+1 >0.$$ 
Also, the conditions $r \geq h_0q+2$ and $h_0\leq \sum_{i=1}^{e-2}q^i$ imply that 
\begin{align*}
    A_{h_0+1} &\leq (h_0+1)\sum_{i=0}^{e-1}q^i - (h_0q+2) \sum_{i=0}^{e-2}q^i\\
    &= h_0+ q^{e-1}- \sum_{i=0}^{e-2}q^i\\
    &\leq q^{e-1}-1=N. 
    \end{align*}
    Hence  $S_N=\{A_{h_0+1}\}$.
Write  $N=(q-1)q^{e-2}+\cdots+(q-1)q+ q-1$.By Lucas' Theorem and~\Cref{binom}(i), we obtain that  $$\sum_{A \in S_N} \binom{N}{A}a^{N-A} =\binom{N}{A_{h_0+1}}a^{N-A_{h_0+1}}\neq 0.$$
Therefore $f(x)$ does not permute $\mathbb F_{q^e}$ for $1 \leq r < \ell$. 

For $r=\ell$, let $j= N/(q-1)$. Then $A_j=0, \ S_N = \{0\}$, $\sum_{A \in S_N} \binom{N}{A}a^{N-A}$  $= \binom{N}{0}a^{N} \neq 0$, and $f(x)$ does not permute $\mathbb F_{q^e}$. Since $q-1\mid N$, by~\Cref{rlargerindex} $f(x)$ does not permute $\fqe$ when $r > \ell$. This completes the proof.
\end{proof}


The next two propositions give more precise descriptions of $r \pmod{\ell}= hq+1$ for which $f(x)=x^r(x^{q-1}+a)$ is a permutation binomial by providing conditions on the divisibility of $h$. We begin by showing that for $e$ even we must have  $q+1\mid h$.

\begin{prop} \label{qplusone4}
Let $f(x) = x^{r}(x^{q-1} + a) \in \mathbb{F}_{q^e}[x]$ with $e$ even and $a\neq 0$,  and let $\ell = q^{e-1}+\cdots + q + 1$. If $r \pmod{\ell} = hq+1$ with $q+1 \nmid h$, then $f(x)$ does not permute $\mathbb{F}_{q^e}$.
\end{prop}

\begin{proof}
We assume that $1 \leq r < \ell$ and  $r=hq+1$, where $h = k(q+1) + h_0$ with positive integers $k$ and $h_0$,  and $1 \leq h_0 \leq q$. Let $N = (q-1) \sum_{i=0}^{(e-2)/2}q^{2i}$. Since $N<\ell$,  \Cref{count} implies that $S_N$ has at most one element. We compute 
\begin{align*}
A_{kq+h_0} = &\left(kq+h_0\right)\sum_{i=0}^{e-1}q^i-\left[\left(k(q+1)+h_0\right)q+1\right]\sum_{i=0}^{(e-2)/2}q^{2i}\\
=& kq\sum_{i=0}^{e-1}q^i+h_0 \sum_{i=0}^{e-1}q^i -kq^2\sum_{i=0}^{(e-2)/2}q^{2i} -kq\sum_{i=0}^{(e-2)/2}q^{2i} \\
 &-h_0q\sum_{i=0}^{(e-2)/2}q^{2i}-\sum_{i=0}^{(e-2)/2}q^{2i}\\
=& k\left(\sum_{i=1}^{e}q^i-\sum_{i=1}^{e/2}q^{2i}-\sum_{i=0}^{(e-2)/2}q^{2i+1}\right) +h_0\left(\sum_{i=0}^{e-1}q^i-\sum_{i=0}^{(e-2)/2}q^{2i+1}\right)-\sum_{i=0}^{(e-2)/2}q^{2i}\\
=&(h_0-1) \sum_{i=0}^{(e-2)/2}q^{2i},
\end{align*}

\noindent
so $0 \leq A_{hq+h_0} \leq N$, and thus $S_N=\{A_{kq+h_0}\}$. By Lucas' Theorem and ~\Cref{binom}(i),  we conclude that
\begin{align*}
    \sum_{A \in S_N} \binom{N}{A}a^{N-A} = \binom{N}{A_{kq+h_0}}a^{N-A_{kq+h_0}} \ne 0.
\end{align*}
Therefore $f(x)$ does not permute $\mathbb F_{q^e}$. Since $q-1\mid N$, we use~\Cref{rlargerindex} to obtain the desired result when $r\geq \ell$. 
\end{proof}

We now show that if $q$ is a power of an odd prime $p$ and $f(x)$  permutes $\fqe$, then $p\mid h$.

\begin{prop}
\label{sth}
Let $f(x) = x^r(x^{q-1} + a)\in\fqe[x]$ with $e\geq 2$, $q$ a power of an odd prime $p$, and $a\neq 0$, 
and let $\ell=q^{e-1}+\cdots+q+1$. If $r \pmod{\ell}= hq+1$  and $p \nmid h$, then $f(x)$ does not permute $\mathbb{F}_{q^e}$.
\end{prop}

\begin{proof} We start by assuming that $r < \ell$ and $r = hq+1$,  where $0 \leq h < \sum_{i=0}^{e-2}q^i$ and $p \nmid h$. Let $N = 2q^{e-1} - 2$. By~\Cref{count}, since $N/\ell=2(q^{e-1}-1)/(\sum_{i=0}^{e-1}q^i) <2$,   we have $|S_N|\leq 2$. We claim that $A_{2h+1}$ is the only element in $S_N$. To show this, we compute
\begin{align*}A_{2h+1} &=(2h+1)\sum_{i=0}^{e-1}q^i - 2(hq+1)\sum_{i=0}^{e-2}q^i \\
&= q^{e-1} -\sum_{i=0}^{e-2}q^i + 2h.
\end{align*}
\noindent
Since $ h \geq 0$, we obtain that  $$A_{2h+1} \geq q^{e-1} -\sum_{i=0}^{e-2}q^i= q^{e-1} -\frac{q^{e-1}-1}{q-1}  >0.$$
\noindent
On the other hand, the condition $h \leq \sum_{i=0}^{e-2}q^i -1$ implies that  
\begin{align*}
A_{2h+1} &\leq q^{e-1}-\sum_{i=0}^{e-2}q^i +2\left(\sum_{i=0}^{e-2}q^i -1\right)\\
&= q^{e-1} + \sum_{i=0}^{e-2}q^i -2\\
&< q^{e-1}+ q^{e-1} -2=N.
\end{align*}
\noindent
Hence $A_{2h+1} \in S_N$. We observe that $A_{j+ k}=A_{j}+k\ell$ for any integers $j$ and $k$, so it remains to show that $A_{2h}$ and $A_{2h+2}$ are not in $[0,N]$. In fact,  
\begin{align*}
    A_{2h}& =A_{2h+1}-\ell\\
& =q^{e-1} -\sum_{i=0}^{e-2}q^i + 2h -\sum_{i=0}^{e-1}q^i \\
& < q^{e-1} -\sum_{i=0}^{e-2}q^i + 2\left(\sum_{i=0}^{e-2}q^i -1\right) -\sum_{i=0}^{e-1}q^i=-2
 \end{align*}
and
$$A_{2h+2}=A_{2h+1}+\ell=q^{e-1} -\sum_{i=0}^{e-2}q^i + 2h +\sum_{i=0}^{e-1}q^i = 2q^{e-1}  + 2h >N.$$ 
Therefore $S_N=\{A_{2h+1}\}$.

Write  $N=q^{e-1}+(q-1)q^{e-2}+\cdots+(q-1)q+(q-2)$ and $A_{2h+1}=a_{e-1}q^{e-1}+\cdots+a_1q+a_0$ in base $q$. Since $p$ is an odd prime and $p \nmid h$, it follows that $A_{2h+1} \equiv 2h - 1 \not\equiv p-1 \pmod{p}$ so that $\binom{q-2}{a_0} \not\equiv 0 \pmod{p}$.  By Lucas' Theorem and~\Cref{binom}, we get 
\begin{align*}
\sum_{A \in S_N} \binom{N}{A}a^{N-A} & =\binom{N}{A_{2h+1}}a^{N-A_{2h+1}} \\
& =\binom{1}{a_{e-1}}\binom{q-1}{a_{e-2}}\cdots\binom{q-1}{a_1}\binom{q-2}{a_0}a^{N-A_{2h+1}}\neq 0.
\end{align*}
Therefore, if $r<\ell$, the polynomial $f(x)$ does not permute $\mathbb F_{q^e}$. Since $q-1\mid N$, we can use~\Cref{rlargerindex} in the case $r\geq \ell$ and conclude the proof. 
\end{proof}


We now consider  $r = \left[h\left(\sum_{i=0}^{e-3}q^i\right)\right]q+1=    h \left(\sum_{i=1}^{e-2}q^i \right) + 1$ for $1 \leq h \leq q-1$. In this case, we obtain a nonexistence result by considering $N = 2q^{e-1}-q^{e-2}-1$. The next result provides the exact elements that belong to $S_N$ depending on $h$.

\begin{lemma}\label{Sn1}
Let  $f(x) = x^r(x^{q-1} + a)\in\mathbb{F}_{q^e}[x]$ with $e,q \geq 3$ and $a\neq 0$, and let $\ell=q^{e-1}+\cdots+q+1$.
 Suppose $r= h (\ell-q^{e-1}-1)  + 1$ for some $1 \leq h \leq q-1$. For $N = 2q^{e-1}-q^{e-2}-1$ and  $j=2hq^{e-3} + h\sum_{i=0}^{e-4}q^i$, we have
\begin{enumerate}[\normalfont(i)]
\item $A_j= (h-2)q^{e-2}+(2h-1)q^{e-3}+(h-1)\displaystyle\sum_{i=0}^{e-4}q^i$,
\item 
$A_{j+1}= q^{e-1}+(h-1)q^{e-2}+2hq^{e-3}+h\displaystyle\sum_{i=0}^{e-4}
q^i$,
\item $S_N=
\begin{cases}
\{A_{j+1}\} &\text{ if } h=1 \\
\{ A_j,A_{j+1}\} & \text{ if } 2\leq h \leq q-2 \\
\{A_{j}\} &\text{ if } h=q-1.\\
\end{cases}$
\end{enumerate}
\end{lemma}

\begin{proof}
We start by showing (i) and (ii). By using that $N/(q-1) = 2q^{e-2}+\sum_{i=0}^{e-3}
q^i$,  we  compute
\begin{align*}
A_j =&\displaystyle\left(2hq^{e-3} + h\sum_{i=0}^{e-4}
q^i\right)\left(\sum_{i=0}^{e-1}
q^i\right)- \left(h\left(\sum_{i=1}^{e-2}
q^i\right)+1\right)\left(2q^{e-2}+\sum_{i=0}^{e-3}
q^i\right)\\
= &2h\sum_{i=e-3}^{2e-4}q^i + h \sum_{i=0}^{e-4}\sum_{j=0}^{e-1}q^{i+j}
-2h \sum_{i=e-1}^{2e-4}q^i - h \sum_{i=1}^{e-2}\sum_{j=0}^{e-3}q^{i+j}
-2q^{e-2}-\sum_{i=0}^{e-3}q^i \\
= &2hq^{e-3} + 2(h-1)q^{e-2}-\sum_{i=0}^{e-3}q^i + h \sum_{i=0}^{e-1}q^i \\
 &+ h \sum_{i=e}^{2e-5}q^{i} +  h \sum_{i=e-1}^{2e-6}q^{i}-h \sum_{i=e-2}^{2e-5}q^i - h \sum_{i=e-3}^{2e-6}q^{i}\\
= &2hq^{e-3} + 2(h-1)q^{e-2}-\sum_{i=0}^{e-3}q^i + h \sum_{i=0}^{e-1}q^i -  hq^{e-1} - 2hq^{e-2} -hq^{e-3}\\
= & (h-2)q^{e-2}+hq^{e-3}+ (h-1)\sum_{i=0}^{e-3}q^i\\
= & (h-2)q^{e-2}+(2h-1)q^{e-3}+ (h-1)\sum_{i=0}^{e-4}q^i.
\end{align*}

We also compute
\begin{align*}
A_{j+1} = & A_j + \ell = A_j+ \sum_{i=0}^{e-1}q^i  \\
= & (h-2)q^{e-2}+(2h-1)q^{e-3}+ (h-1)\sum_{i=0}^{e-4}q^i + \sum_{i=0}^{e-1}q^i\\
= & q^{e-1}+(h-1)q^{e-2}+2hq^{e-3}+h\sum_{i=0}^{e-4} q^i.
\end{align*}

We now show (iii). By~\Cref{count}, since $\ell = \sum_{i=0}^{e-1}q^i$ and $N/\ell=(2q^{e-1}-q^{e-2}-1)/\ell<2$, we have $|S_N|\leq 2$. The elements in $S_N$ satisfy $A_{i\pm 1} = A_i\pm \ell$; in particular,  $A_i<A_k$ whenever $i<k$. Next we describe the elements in $S_N$ depending on $h$.

\

{\underline{Case 1}}: $h=1$ 

Since $A_j=-q^{e-2}+q^{e-3}$ is negative, it does not belong to $S_N$.
Let us consider $A_{j+1}$. We have $A_{j+1}=q^{e-1}+2q^{e-3}+\sum_{i=0}^{e-4} q^i \geq 0$ and
\begin{align*}
 & A_{j+1}  \leq  2q^{e-1}-q^{e-2}-1 =N\\
 \Longleftrightarrow \ & 2q^{e-3} + \sum_{i=0}^{e-4} q^i \leq q^{e-1} -q^{e-2}-1 \\
 \Longleftrightarrow \ & 2q^{e-3} + \sum_{i=0}^{e-4} q^i \leq (q-1)\sum_{i=0}^{e-2}q^i -q^{e-2} \\
 \Longleftrightarrow \ & 2q^{e-3} + \sum_{i=0}^{e-4} q^i \leq (q-2)q^{e-2}+  (q-1)\sum_{i=0}^{e-3}q^i, \\
\end{align*}
\noindent
which is true,  since $2q^{e-3} \leq (q-2)q^{e-2}$ and $\sum_{i=0}^{e-4} q^i \leq (q-1)\sum_{i=0}^{e-3}q^i$, for $q\geq 3$. Hence   $A_{j+1}\in S_N$. Since
$$
 A_{j+2}  =A_{j+1}+\ell
 =2q^{e-1}+q^{e-2}+3q^{e-3}+2\sum_{i=0}^{e-4} q^i > N,
$$
we conclude that $S_N=\{A_{j+1}\}$.

\

{\underline{Case 2}}: $2\leq h\leq q-2$

We have
\begin{align*}
    0&\leq 3q^{e-3} + \sum_{i=0}^{e-4}q^i \\
    &\leq A_j <A_{j+1}\\
    &\leq 2q^{e-1} -q^{e-2} -3q^{e-3} -\sum_{i=1}^{e-4}q^i -2 \\
    &\leq  2q^{e-1}-q^{e-2}-1 = N,
\end{align*}
which shows that $S_N=\{A_j,A_{j+1}\}$.

\

{\underline{Case 3}}: $h=q-1$ 

In this case, we have
$A_j= q^{e-1}-q^{e-2}-2q^{e-3}-\sum_{i=0}^{e-4}q^i-1$. So
\begin{align*}
A_{j-1} &=A_j-\sum_{i=0}^{e-1}q^i=-2q^{e-2}-3q^{e-3}-2\sum_{i=0}^{e-4}q^i-1<0,\\ 
0 & \leq A_j \leq 2q^{e-1}-q^{e-2}-1= N,\quad\text{ and }\\
A_{j+1}& = 2q^{e-1}-q^{e-3}-1 > 2q^{e-1}-q^{e-2}-1=N.
\end{align*}
Hence $S_N=\{A_j\}$.
\end{proof}

\begin{prop} \label{generalLemma}
Let  $f(x) = x^r(x^{q-1} + a)\in\mathbb{F}_{q^e}[x]$ with $e,q\geq 3$, $q$ a power of a prime $p$, $a\neq 0$,  and  $(-a)^{\ell} \ne 1$, and let $\ell=q^{e-1}+\cdots+q+1$.  Suppose that $r \pmod{\ell} = h \left(\ell-q^{e-1}-1 \right) + 1$ for some $1 \leq h \leq q-1$. Then $f(x)$ does not permute $\mathbb{F}_{q^e}$. 
\end{prop}

\begin{proof} By~\Cref{sth}, if $q$ is odd and $p\nmid h \left(\sum_{i=0}^{e-3}q^i\right)$, then $f(x)$ does not permute $\mathbb{F}_{q^e}$. Hence  $p\mid h$, if $q$ is odd.  Assume $r < \ell$. We consider the following three cases: $h=1$, $2 \leq h \leq q-2$, and $h=q-1$. Let $N=2q^{e-1}-q^{e-2}-1$ whose $q$-expansion is given by  $q^{e-1}+(q-2)q^{e-2}+(q-1)\sum_{i=0}^{e-3}q^{i}$. \Cref{Sn1} provides the exact elements in $S_N$ in each case.

\
 
{\underline{Case 1}}:  $h=1$ 

\

In this case,  $S_N=\{A_{j+1}\}$ and $A_{j+1}=q^{e-1}+2q^{e-3}+\sum_{i=0}^{e-4} q^i$ is the $q$-adic expansion of $A_{j+1}$. We compute 
\begin{align*}
 \sum_{A \in S_N}\binom{N}{A}a^{N-A}
 & = \binom{N}{A_{j+1}}a^{N-A_{j+1}} \\
 & = \binom{1}{1} \binom{q-2}{0} \binom{q-1}{2} \binom{q-1}{1}^{e-3}a^{N-A_{j+1}},
\end{align*}
which is not zero by~\Cref{binom}(i).

\

{\underline{Case 2}}: $2 \leq h \leq q-2$ 

\

We have $S_N=\{A_j,A_{j+1}\}$. 

\

{\underline{Subcase 2.1}}:  $2\leq h < \left\lfloor q/2\right\rfloor$ 

\

In this case,
\begin{align*}
    A_j&=(h-2)q^{e-2}+(2h-1)q^{e-3}+ (h-1)\sum_{i=0}^{e-4}q^i\quad\text{ and }\\  A_{j+1}&=q^{e-1}+(h-1)q^{e-2}+2hq^{e-3}+h\sum_{i=0}^{e-4} q^i,
    \end{align*}
so
\begin{align*} 
\sum_{A \in S_N}\binom{N}{A}a^{N-A} & = \binom{N}{A_j}a^{N-A_j} + \binom{N}{A_{j+1}}a^{N-A_{j+1}}\\  
    & = \binom{1}{0} \binom{q-2}{h-2} \binom{q-1}{2h-1} \binom{q-1}{h-1}^{e-3}a^{N-A_j}\\
    &\quad + \binom{1}{1} \binom{q-2}{h-1} \binom{q-1}{2h} \binom{q-1}{h}^{e-3}a^{N-A_{j+1}}.\end{align*}

 By~\Cref{binom}(ii), when $q$ is even, $\binom{q-2}{k}= 0$ if and only if  $k$ is odd.  Also, if $q$ is odd, $p\mid h$ and we get $\binom{q-2}{h-1} \equiv 0 \pmod{p}$. Therefore
$$
      \sum_{A \in S_N}\binom{N}{A}a^{N-A} 
=\begin{cases}
    \displaystyle\binom{q-2}{h-1} \binom{q-1}{2h} \binom{q-1}{h}^{e-3}a^{N-A_{j+1}} \ne 0, & \text{ if } h \text{ is odd and } q \text{ is even} \\ 
    \displaystyle\binom{q-2}{h-2} \binom{q-1}{2h-1} \binom{q-1}{h-1}^{e-3}a^{N-A_j} \ne 0, & \text{ otherwise.}
    \end{cases}
$$
\

{\underline{Subcase 2.2}}: $q$ is even and $h = q/2$ 

\
    
    In this case, the $q$-adic expansions of the two elements in $S_N$ are 
    \begin{align*}
    A_j &= \left(\dfrac{q}{2} -2 \right)q^{e-2} + (q-1)q^{e-3} + \left(\dfrac{q}{2}-1 \right) \sum_{i=0}^{e-4}q^i  \quad\text{ and }\\
    A_{j+1} &= q^{e-1} + \left(\dfrac{q}{2}\right)q^{e-2} + \dfrac{q}{2} \sum_{i=0}^{e-4}q^i.
    \end{align*}
    Consequently, 
\begin{align*} 
    \sum_{A \in S_N}\binom{N}{A}a^{N-A} & =\binom{N}{A_j}a^{N-A_j} + \binom{N}{A_{j+1}}a^{N-A_{j+1}} \\
    & = \binom{1}{0} \binom{q-2}{q/2 -2} \binom{q-1}{q-1} \binom{q-1}{q/2 -1}^{e-3}a^{N-A_{j+1}+\ell}\\
    &\quad + \binom{1}{1} \binom{q-2}{q/2} \binom{q-1}{0} \binom{q-1}{q/2}^{e-3}a^{N-A_{j+1}} \\
 & = \binom{q-2}{q/2} \binom{q-1}{q/2}^{e-3}a^{N-A_{j+1}} \left(a^{\ell}  + 1 \right).
\end{align*}
Since $\ell$ is odd and $(-a)^{\ell}\neq 1$, we have $a^{\ell}  + 1 \ \neq 0$.  We conclude that the sum is not zero by applying~\Cref{binom}(i) and (ii).

\

{\underline{Subcase 2.3}}: $\lfloor q/2 \rfloor < h \leq q-2$

\

Write $h = \lfloor q/2 \rfloor + h_0$ with $h_0 >0$. 

If $q$ is even, then $h=q/2+h_0$, $0 < h_0 \leq (q-4)/2$ and  the $q$-adic expansions of $A_j$ and $A_{j+1}$ are
\begin{align*}
    A_j &= (h-1)q^{e-2} + (2h_0-1)q^{e-3} + \sum_{k=0}^{e-4}(h-1)q^k \quad\text{ and }\\
A_{j+1} &= q^{e-1}+hq^{e-2} + 2h_0q^{e-3} + h\sum_{k=0}^{e-4}q^k.
\end{align*}
Thus
\begin{align*} 
    \sum_{A \in S_N}\binom{N}{A}a^{N-A} & = \binom{N}{A_j}a^{N-A_j} + \binom{N}{A_{j+1}}a^{N-A_{j+1}}  \\
    & = \binom{1}{0} \binom{q-2}{h-1} \binom{q-1}{2h_0-1} \binom{q-1}{h-1}^{e-3}a^{N-A_j}\\
    &\quad + \binom{1}{1} \binom{q-2}{h} \binom{q-1}{2h_0} \binom{q-1}{h}^{e-3}a^{N-A_{j+1}}\\
     & =
    \begin{cases}
    \displaystyle\binom{q-2}{h} \binom{q-1}{2h_0} \binom{q-1}{h}^{e-3}a^{N-A_{j+1}} \ne 0, & \text{ if } h \text{ even }  \\ 
    \displaystyle\binom{q-2}{h-1} \binom{q-1}{2h_0-1} \binom{q-1}{h-1}^{e-3}a^{N-A_j} \ne 0, & \text{ otherwise}.
    \end{cases}
\end{align*}

If $q$ is odd, then $h=(q-1)/2+h_0$ with $0 < h_0 \leq (q-3)/2$, and the $q$-adic expansions of $A_j$ and $A_{j+1}$ are 
\begin{align*}
    A_j &= (h-1)q^{e-2} + (2h_0-2)q^{e-3} + (h-1)\sum_{i=0}^{e-4}q^i\quad\text{ and }\\
    A_{j+1} &= q^{e-1} + hq^{e-2} + (2h_0-1)q^{e-3} + h\sum_{i=0}^{e-4}q^i.
    \end{align*}
This implies that
\begin{align*} 
    \sum_{A \in S_N}\binom{N}{A}a^{N-A} & = \binom{N}{A_j}a^{N-A_j} + \binom{N}{A_{j+1}}a^{N-A_{j+1}} \\
     & = \binom{1}{0} \binom{q-2}{h-1} \binom{q-1}{2h_0-2} \binom{q-1}{h-1}^{e-3}a^{N-A_j}\\
     &\quad + \binom{1}{1} \binom{q-2}{h} \binom{q-1}{2h_0-1} \binom{q-1}{h}^{e-3}a^{N-A_{j+1}}.
    \end{align*}
    By using~\Cref{binom}(i) and (ii), since $p\mid h$, we obtain that this sum  is
     $$ \binom{q-2}{h} \binom{q-1}{2h_0-1} \binom{q-1}{h}^{e-3}a^{N-A_{j+1}} \ne 0.$$

\

\underline{Case 3}: $h= q-1$ 

\

In this case, $S_N=\{A_j\}$ and the $q$-adic expansion of $A_j$ is $(q-2)q^{e-2}+(q-3)q^{e-3}+(q-2)\sum_{i=0}^{e-4}q^i$. Then
\begin{align*}
    \sum_{A \in S_N}\binom{N}{A}a^{N-A} = \binom{N}{A_j}a^{N-A_j} = \binom{1}{0} \binom{q-2}{q-2} \binom{q-1}{q-3} \binom{q-1}{q-2}^{e-3}a^{N-A_j},
\end{align*}
which is not zero by~\Cref{binom}(i).

\

Therefore, if $r<\ell$, the polynomial $f(x)$ does not permute $\mathbb F_{q^e}$. Since $N/(q-1)=2q^{e-2}+\sum_{i=1}^{e-3}q^i$  is an integer, we use~\Cref{rlargerindex} for the case $r\geq \ell$ and conclude the proof. 
\end{proof}

\section{Characterizations}\label{special}

Permutation binomials of the form $f(x)=x^r(x^{q-1}+a)$ over $\fqe$  were studied by Hou  when $e=2$, and by Liu when $e=3$ with $q$ odd and  $1\leq r \leq \ell$, where $\ell=q^{e-1}+\cdots + q+1$; see~\Cref{liliu}.  In this section we prove~\Cref{e2345} by applying our results from~\Cref{suff,nece} to present complete characterizations of these permutation binomials   over $\f_{q^2}$, $\f_{q^3}$, $\f_{q^4}$, $\f_{p^5}$, and $\f_{p^6}$, where $p$ is an odd prime, for all values of $r$. As we mentioned earlier, two necessary conditions are that $(-a)^{\ell}\neq 1$ and $\gcd(r,q-1)=1$. We make use of these two facts without any further reference. In addition, as in~\Cref{nece}, we show our results for $e\in\{5,6\}$ by using~\Cref{MPWbin} and Lucas' Theorem without referring to them.

\subsection{Fields $\f_{q^2}, \f_{q^3}$, and $\f_{q^4}$}\label{234}

We extend~\Cref{liliu} in the following directions: from $r\in\{1,\dots,\ell\}$ to $r\geq 1$, from $q$ odd and $e=3$ to any $q$ and $e=3$, and  from $e\in\{2,3\}$ to $e\in\{2,3,4\}$. Therefore our  characterization  provides a full description of permutation binomials $f(x) = x^r(x^{q-1} + a)$ over $\fqe$  for $e\in\{2,3,4\}$. For these values of $e$, the description of $r$ written in terms of $\ell$ turns out to be the same.

\begin{prop} \label{q234final}
Let $f(x) = x^r(x^{q-1} + a) \in \mathbb{F}_{q^e}[x]$ with $e\in\{2,3,4\}$ and  $a\neq 0$, and let $\ell=q^{e-1}+\cdots + q+1$. Then $f(x)$ permutes $\mathbb{F}_{q^e}$ if and only if  $(-a)^{\ell} \ne 1$, $\gcd(r, q-1) = 1$, and $r\pmod{\ell}\in\{ 1,\ell-q\}$.
\end{prop}

\begin{proof}
By~\Cref{part}, if $(-a)^{\ell} \ne 1$, $r \pmod{\ell} \in \{1, \ell-q\}$, and $\gcd(r, q-1) = 1$,   then $f(x)$ permutes $\mathbb{F}_{q^e}$. Conversely, suppose that $f(x)$ permutes $\mathbb{F}_{q^e}$ for $e \in \{2,3, 4\}$. By~\Cref{rth}, $r\pmod\ell = hq+1$ for some integer $h$.

When $e=2$, since $\ell=q+1$, we must have  $r\pmod\ell=1$. 

When $e=3$, we have  $\ell=(q+1)q+1$, so $0\leq h \leq q$. \Cref{generalLemma} implies that  $h \in \{0, q\}$, so $r \pmod{\ell} \in \{1, \ell-q\}$.

When $e=4$, we have  $\ell=(q^2+q+1)q+1$. Combining Propositions~\ref{rth} and~\ref{qplusone4} gives that $r \pmod{\ell} = k(q+1)q + 1$ with $0 \leq k \leq q$. \Cref{generalLemma} implies that $k \in \{0, q\}$, so $r \pmod{\ell} \in \{1, \ell-q\}$.
\end{proof}

For fields $\f_{q^2}$, we note that $r \pmod{\ell}=\ell-q=1$.

\subsection{Fields \texorpdfstring{$\f_{p^5}, \ p$}{Fp5, p} an odd prime} \label{5}

In~\Cref{q234final} we show that the only values of $r\pmod\ell$ such that  $x^r(x^{q-1}+a)$ permutes $\fqe$   for $e\in\{2,3,4\}$ and $a\neq 0$ are  $1$ and $\ell-q$, with  $\gcd(r,q-1)=1$. In $\f_{p^5}[x]$  there are two additional values. 

\begin{prop}\label{p5}
Let $f(x) = x^r(x^{p-1} + a) \in \mathbb{F}_{p^5}[x]$ with $p$ an odd prime and $a\neq 0$, and let $\ell=p^4+p^3+p^2+p+1$. Then $f(x)$ permutes $\mathbb{F}_{p^5}$ if and only if $(-a)^{\ell} \ne 1$, $\gcd(r, p-1) = 1$, and $r \pmod{\ell}\in\{1,  p^3+1,p^4+p^2+1, \ell-p\}$.
\end{prop}

\begin{proof}
By~\Cref{part}, if $(-a)^{\ell} \ne 1$, $r \pmod{\ell} \in \{1, p^3+1,p^4+p^2+1, \ell-p\}$, and $\gcd(r, p-1) = 1$,
 then   $f(x)$ permutes $\mathbb{F}_{p^5}$. Conversely, let $f(x)$ be a permutation binomial over $\f_{p^5}$. \Cref{rth} and~\Cref{sth} imply that $r \pmod{\ell}$ has the form $hp^2 + 1$ for some $0 \leq h \leq p^2+p+1$. We show that we cannot have the following three cases for $h$:
 $1 \leq h \leq p-1$, $p+1 \leq h \leq p^2$, and $p^2 +2 \leq h \leq p^2+p$.

{\underline{Case 1}}: $1 \leq h \leq p-1$ 

Consider $N = p^4-p^2-p+1$ whose $p$-adic expansion is $(p-1)p^3+(p-2)p^2+(p-1)p+1$. By~\Cref{count}, since $N<\ell$, the set $S_N$ has at most one element. We compute
\begin{align*}
A_{hp} &= hp(p^4+p^3+p^2+p+1)-\dfrac{(hp^2+1)(p^4-p^2-p+1)}{p-1}\\
&= hp(p^4+p^3+p^2+p+1)-(hp^2+1)(p^3+p^2-1)\\
&=(h-1)p^3 + (2h-1)p^2 + hp + 1.  \end{align*}
Since $1\leq h\leq p-1$, we have $0 < A_{hp} \leq p^4-2p^2-p+1 <N$. Hence $A_{hp}$ is the unique element in  $S_N$.

{\underline{Subcase 1.1}}: $1 \leq h \leq (p-1)/2$

By using Lucas' Theorem and~\Cref{binom}(i) and (ii), we have
\begin{align*}
    \sum_{A \in S_N} \binom{N}{A}a^{N-A} & = \binom{N}{A_{hp}}a^{N-A_{hp}}\\
    & = \binom{p-1}{h-1} \binom{p-2}{2h-1} \binom{p-1}{h} \binom{1}{1}a^{N-A_{hp}}\neq 0.
\end{align*}

{\underline{Subcase 1.2}}: $(p-1)/2 < h \leq p-1$ 

Write $h = (p-1)/2 + h_0$ with $1\leq  h_0\leq (p-1)/2$. Then \begin{align*}
  A_{hp} & =  (h-1)p^3 + \left(2\left(\dfrac{p-1}{2} + h_0\right)-1\right)p^2 + hp + 1\\
&=hp^3 + (2h_0 -2)p^2 + hp + 1.
\end{align*}  
Since $0\leq 2h_0-2 \leq p-3$, we have

\begin{align*}
    \sum_{A \in S_N} \binom{N}{A}a^{N-A} &= \binom{N}{A_{hp}}a^{N-A_{hp}}\\ 
    & = \binom{p-1}{h} \binom{p-2}{2h_0-2} \binom{p-1}{h} \binom{1}{1}a^{N-A_{hp}} \ne 0.
\end{align*}

{\underline{Case 2}}: $p+1 \leq h \leq p^2$ 

Consider $N = p^3-1$ whose $p$-adic expansion is $(p-1)p^2+(p-1)p+p-1$. By~\Cref{count}, since $N<\ell$,  the set $S_N$ has at most one element.  We compute
\begin{align*}
    A_{h} &= h(p^4+p^3+p^2+p+1)-\dfrac{(hp^2+1)(p^3-1)}{p-1}\\
&= h(p^4+p^3+p^2+p+1)-(hp^2+1)(p^2+p+1)\\
&=hp+h-p^2-p-1. 
\end{align*}
Since $p+1 \leq h \leq p^2,$ we have $0 <  A_h\leq p^3-p-1 <N$,
so $A_h$ is the unique element  in $S_N$. Let $A_h=a_2p^2 + a_1p + a_0$ be the $p$-adic expansion of $A_h$. Since $a_0, a_1, a_2 \leq p-1$,  Lucas' Theorem and~\Cref{binom}(i) imply that

\begin{align*}
    \sum_{A \in S_N} \binom{N}{A}a^{N-A} = \binom{N}{A_{h}}a^{N-A_{h}} = \binom{p-1}{a_2} \binom{p-1}{a_1} \binom{p-1}{a_0} a^{N-A_{h}} \ne 0.
\end{align*}

\underline{Case 3}:  $p^2+2 \leq h \leq p^2+p$ 

As in Case 1, consider $N = p^4-p^2-p+1=(p-1)p^3+(p-2)p^2+(p-1)p+1$. Then $S_N$ has at most one element. We compute
\begin{align*}
    A_{hp-p-1}
    &=(hp-p-1)(p^4+p^3+p^2+p+1)-\dfrac{(hp^2+1)(p^4-p^2-p+1)}{p-1}\\
&= (hp-p-1)(p^4+p^3+p^2+p+1)-(hp^2+1)(p^3+p^2-1)\\
    &=-p^5-2p^4+(h-3)p^3+(2h-3)p^2+(h-2)p.
\end{align*}
\noindent For $p^2+2 \leq h \leq p^2+p,$ we have
$ 0 < p^2   \leq A_{hp-p-1} \leq p^4-2p^2-2p < N,$
so $A_{hp-p-1}$ is the unique element in $S_N$.
Write $h = p^2+2 + h_0$ with $0 \leq h_0 \leq p-2$. Then $$A_{hp-p-1}= h_0p^3+(2h_0+1)p^2+h_0p.$$

\underline{Subcase 3.1}: $0 \leq h_0 \leq (p-3)/2$ 

Since $2h_0+1 \leq p-2$, if follows that
\begin{align*}
    \sum_{A \in S_N} \binom{N}{A}a^{N-A}& = \binom{N}{A_{hp-p-1}}a^{N-A_{hp-p-1}}\\
    &= \binom{p-1}{h_0} \binom{p-2}{2h_0+1} \binom{p-1}{h_0} \binom{1}{0}a^{N-A_{hp-p-1}} \neq 0.
\end{align*}

\underline{Subcase 3.2}: $(p-3)/2 < h_0 \leq p-2$ 

Write $h_0 = (p-3)/2 + h_1$ with $1 \leq h_1 \leq (p-1)/2$. Then $$A_{hp-p-1}= (h_0+1)p^3+(2h_1-2)p^2+h_0p.$$
and, since $2h_1-2 \leq p-3$, we have

\begin{align*}
    \sum_{A \in S_N} \binom{N}{A}a^{N-A} & = \binom{N}{A_{hp-p-1}}a^{N-A_{hp-p-1}}\\ 
    & = \binom{p-1}{h_0+1} \binom{p-2}{2h_1-2} \binom{p-1}{h_0} \binom{1}{0}a^{N-A_{hp-p-1}} \neq 0.
\end{align*}

\noindent
By~\Cref{MPWbin}, $f(x)$ cannot permute $\f_{p^5}$ with $r \pmod{\ell} = hp^2+1$ for $1 \leq h \leq p-1$, $p+1 \leq h \leq p^2$, and $p^2 +2 \leq h \leq p^2+p$, and this contradicts our assumption. The remaining values of $h$ are $0, p, p^2+1, p^2+p+1$, hence necessarily 
$r \pmod{\ell}\in \{1,  p^3+1, p^4+p^2+1, \ell-p \}$. This completes the proof.
\end{proof}

\subsection{Fields \texorpdfstring{$\mathbb{F}_{p^6}$, $p$}{Fp6, p} an odd prime} \label{6}

The characterization of  permutation binomials of the form $f(x) = x^r(x^{p-1} + a)$ over $\f_{p^6}$ is the same one as for $\mathbb{F}_{q^2}$, $\mathbb{F}_{q^3}$, and $\mathbb{F}_{q^4}$ from~\Cref{q234final}.

\begin{prop}\label{p6}
 Let $f(x) = x^r(x^{p-1} + a) \in \mathbb{F}_{p^6}[x]$ with $p$ an odd prime and $a\neq 0$, and let $\ell=p^5+p^4+p^3+p^2 +p+1$. Then $f(x)$ permutes $\mathbb{F}_{p^6}$ if and only if  $(-a)^{\ell} \ne 1$, $\gcd(r, p-1) = 1$, and $r \pmod{\ell}\in\{1, \ell- p\}$.
\end{prop}
\begin{proof}
By~\Cref{part}, if $(-a)^{\ell} \ne 1$, $r \pmod{\ell} \in \{1, \ell-p\}$, and $\gcd(r, p-1) = 1$, then $f(x)$ permutes $\mathbb{F}_{p^6}$. Conversely, suppose that $f(x)$ permutes $\f_{p^6}$. By combining Propositions~\ref{rth}-\ref{sth}, we may assume $r \pmod{\ell}= n(p^3 + p^2) + 1$ with $0 \leq n \leq p^2 + 1$.

We will show that we cannot have $1 \leq n \leq p^2$. Let $N = p^5 - p^3 + p^2-1$ with $p$-expansion given by  $(p-1)p^4 + (p-1)p^3 + (p-1)p + p-1$. Then $\lfloor N/\ell \rfloor = 0$. By~\Cref{count}, $\lvert S_N \rvert \leq 1$. Let $j = n(p^2 +p -1) + 1$. We compute 
\begin{align*}
A_j &=[n(p^2 +p -1) + 1)]\ell-\dfrac{[n(p^3 + p^2) + 1](p^5 - p^3 + p^2-1)}{p-1} \\    
    & = [n(p^2 +p -1) + 1]\ell-(np^3 + np^2 + 1)(\ell-p^5-p^2) \\ 
    &= n p^2\ell+np\ell -n\ell+\ell-np^3\ell+np^8+np^5-np^2\ell+np^7+np^4-\ell+p^5+p^2\\
    & = n[p\ell-\ell+ p^3(-\ell+p^5+p^4+p^2+p)]+p^5+p^2\\
  &=  n(-p^3-1)+p^5+p^2\\
     &=p^5 - np^3+ p^2-n.
\end{align*}
Since $1\leq n\leq p^2$, we have
$0\leq A_j\leq N,$
so $S_N=\{A_j\}$.

\underline{Case 1}: $1 \leq n < p^2$ 

Write $n = n_1p + n_0$, where $0 \leq n_0, n_1 \leq p-1$. If $n_0 \neq 0$, the $p$-expansion of $A_j$ is 
$$
    A_j = (p-n_1-1)p^4 + (p-n_0)p^3 + (p-n_1-1)p + (p-n_0).
$$
By~\Cref{binom}(i), it turns out that
\begin{equation*}
    \sum_{A \in S_N} \binom{N}{A}a^{N-A} = \binom{p-1}{p-1-n_1} \binom{p-1}{p-n_0} \binom{p-1}{p-1-n_1} \binom{p-1}{p-n_0}a^{N-A_j} \ne 0.
\end{equation*}
\noindent If $n_0 = 0$, then $n_1 \neq 0$ and $A_j = (p-n_1)p^4 + (p-n_1)p$, where $1 \leq p-n_1 \leq p-1$. Thus, 

\begin{equation*}
    \sum_{A \in S_N} \binom{N}{A}a^{N-A} = \binom{p-1}{p-n_1} \binom{p-1}{p-n_1}a^{N-A_j} \ne 0.
\end{equation*}

\underline{Case 2}: 
$n = p^2$

Then $A_j =0$ and $\sum_{A \in S_N} \binom{N}{A}a^{N-A} = \binom{N}{0}a^{N} \neq 0$. 

By~\Cref{MPWbin}, $f(x)$ cannot permute $\f_{p^6}$ with $r \pmod{\ell} = n(p^3+p^2)+1, \ 1 \leq n \leq p^2$, and this contradicts our assumption. Hence  $n \in \{0, p^2+1\}$ and  $r\pmod \ell \in \{1, \ell -p\}$.
\end{proof}


\subsection{Fields $\f_{q^e}$}

We conjecture that all  permutations binomials of  the  form $x^r\left(x^{q-1}+a\right)$ over $\f_{q^e}$  are the ones described in~\Cref{Compos}. More specifically, we have the following.

\bigskip

\begin{conjecture} \label{conjecture}
Let $f(x) = x^r(x^{q-1} + a)\in\mathbb{F}_{q^e}[x]$ with $e\geq 2$ and $a\neq 0$,  and let
$\ell=q^{e-1}+\cdots +q+1$.
Then $f(x)$ permutes $\fqe$ if and only if $f(x)$ is congruent to the composition of a linearized binomial $L(x)=x^{q^h}+ax$ and the monomial $x^r$ modulo $x^{q^e}-x$, where  $\left(-a\right)^{\ell}\neq 1$   and $\gcd(r,q-1)=1$. 
\end{conjecture}

If~\Cref{conjecture} holds true, by~\Cref{Compos} the only values of $r$ for which  $f(x)=x^r\left(x^{q-1}+a\right)$ permutes $\f_{q^e}$ satisfy $r = s\ell + \sum_{i=0}^{k-1}q^{hi}  \pmod{q^e-1}$, where $\gcd(h, e) = 1$, $k \pmod{e} = h^{-1}$, and $s$ is an integer. The results in~\Cref{234,5,6} confirm~\Cref{conjecture} for  fields $ \f_{q^e}$, with $e=2,3,4$,  and $\f_{p^e}$, with $e=5,6$ and an odd prime $p$. An exhaustive search for all permutation binomials of the form $f(x)$ for $q^e < 10^8$ also verified the conjecture. To optimize our search, we used  results from~\Cref{suff,nece,special} along with  the following.

\begin{prop}
Let $\xi$ be a primitive element of $\mathbb{F}_{q^e}^*$, and $f_m(x) = x^r(x^{q-1} + \xi^m) \in \mathbb{F}_{q^e}[x]$. If $f_j(x)$ does not permute $\mathbb F_{q^e}$ for all $0 \leq j < q-1$, then $f_m(x)$ does not permute $\mathbb F_{q^e}$ for all $0 \leq m < q^e-1$.
\end{prop}

\begin{proof}  Let  $0 \leq m < q^e-1$. Write $m = s(q-1) + j$ with  $0 \leq j < q-1$. Since $f_j(x)$ does not permute $\mathbb{F}_{q^e}$, there exist integers $i_0$ and $i_1$   such that $0\leq i_0 < i_1 < q^e-1$ and $f_j(\xi^{i_0}) = f_j(\xi^{i_1})$. Thus, $\xi^{rs + s(q-1)}f_j(\xi^{i_0}) = \xi^{rs + s(q-1)}f_j(\xi^{i_1})$ implies that $\xi^{r(i_0 + s)}(\xi^{(i_0+s)(q-1)} + \xi^m) = \xi^{r(i_1 + s)}(\xi^{(i_1+s)(q-1)} + \xi^m)$. Therefore, $f_m(\xi^{i_0+s}) = f_m(\xi^{i_1+s})$, and $f_m(x)$ does not permute $\mathbb{F}_{q^e}$.
\end{proof}

\bigskip

\section*{Acknowledgments}

The first author received support for this project
provided by  PSC-CUNY grants, awards \#62094-00 50 and \#63541-00 51, jointly funded by The Professional Staff Congress and The City University of New York. The research of the third author was supported in part by the Puerto Rico Louis Stokes Alliance for Minority Participation (PR-LSAMP) program at the University of Puerto Rico, NSF grant \#1400868.
Part of this work was carried out when the third author  visited the first author in Summer 2019. He thanks the support and warm hospitality received at New York City College of Technology.

\bibliographystyle{plainurl}
\bibliography{refs.bib}

\begin{thebibliography}{10}

\bibitem{akbary2009permutation}
A.~Akbary, D.~Ghioca, and Q.~Wang.
\newblock On permutation polynomials of prescribed shape.
\newblock {\em Finite Fields Appl.}, 15(2):195--206, 2009.

\bibitem{hermite}
C.~Hermite.
\newblock Sur les fonctions de sept lettres.
\newblock {\em C. R. Acad. Sci. Paris}, 57:750--757, 1863.

\bibitem{MR3071828}
X.~Hou.
\newblock A class of permutation binomials over finite fields.
\newblock {\em J. Number Theory}, 133(10):3549--3558, 2013.

\bibitem{MR3293406}
X.~Hou.
\newblock Permutation polynomials over finite fields -- a survey of recent
  advances.
\newblock {\em Finite Fields Appl.}, 32:82--119, 2015.

\bibitem{MR3329981}
X.~Hou.
\newblock A survey of permutation binomials and trinomials over finite fields.
\newblock In {\em Topics in finite fields}, volume 632 of {\em Contemp. Math.},
  pages 177--191. Amer. Math. Soc., Providence, RI, 2015.

\bibitem{MR3587259}
X.~Hou.
\newblock Permutation polynomials of {$\mathbb{F}_{q^2}$} of the form
  {$aX+X^{r(q-1)+1}$}.
\newblock In {\em Contemporary developments in finite fields and applications},
  pages 74--101. World Sci. Publ., Hackensack, NJ, 2016.

\bibitem{houarxiv2016}
X.~Hou.
\newblock Permutation polynomials of the form $x^r(a+x^{2(q-1)})$ -- a
  nonexistence result.
\newblock {\em arXiv:1609.03662}, 2016.

\bibitem{MR3276311}
X.~Hou and S.~D. Lappano.
\newblock Determination of a type of permutation binomials over finite fields.
\newblock {\em J. Number Theory}, 147:14--23, 2015.

\bibitem{MR3333142}
S.~D. Lappano.
\newblock A note regarding permutation binomials over {$\mathbb{F}_{q^2}$}.
\newblock {\em Finite Fields Appl.}, 34:153--160, 2015.

\bibitem{LiQuChFFA2017}
K.~Li, L.~Qu, and X.~Chen.
\newblock New classes of permutation binomials and permutation trinomials over
  finite fields.
\newblock {\em Finite Fields Appl.}, 43:69--85, 2017.

\bibitem{LiuArX2019}
X.~Liu.
\newblock Some results about permutation properties of a kind of binomials over
  finite fields.
\newblock {\em arXiv:1906.09168}, 2019.

\bibitem{MPW}
A.~Masuda, D.~Panario, and Q.~Wang.
\newblock The number of permutation binomials over {$\mathbb F_{4p+1}$} where
  {$p$} and {$4p+1$} are primes.
\newblock {\em Electron. J. Combin.}, 13(1):Research Paper 65, 15, 2006.

\bibitem{MaZiTransAMS2009}
A.~M. Masuda and M.~E. Zieve.
\newblock Permutation binomials over finite fields.
\newblock {\em Trans. Amer. Math. Soc.}, 361(8):4169--4180, 2009.

\bibitem{MR3087321}
G.~L. Mullen and D.~Panario, editors.
\newblock {\em Handbook of finite fields}.
\newblock Discrete Mathematics and its Applications (Boca Raton). CRC Press,
  Boca Raton, FL, 2013.

\bibitem{Wa2019}
Q.~Wang.
\newblock Polynomials over finite fields: an index approach.
\newblock {\em in the Proceedings of Pseudo-Randomness and Finite Fields,
  Multivariate Algorithms and their Foundations in Number Theory, October
  15-19, Linz, 2018, Combinatorics and Finite Fields. Difference Sets,
  Polynomials, Pseudorandomness and Applications}, pages 319--348, 2019.

\bibitem{zarxivredei13}
M.~E. Zieve.
\newblock Permutation polynomials on $\mathbb{F}_q$ induced from bijective
  {R}\'edei functions on subgroups of the multiplicative group of $\mathbb
  {F}_q$.
\newblock {\em arXiv:1310.0776}, 2013.

\end{thebibliography}

\end{document}